\documentclass[12pt]{article}%
\usepackage{graphicx}
\usepackage[intlimits]{amsmath}
\usepackage{latexsym}
\usepackage{amsfonts}
\usepackage{amssymb}%
 \makeatletter
\newfont{\footsc}{cmcsc10 at 8truept}
\newfont{\footbf}{cmbx10 at 8truept}
\newfont{\footrm}{cmr10 at 10truept}
\newtheorem{theorem}{Theorem}

\newtheorem{definition}[theorem]{Definition}

\newtheorem{lemma}[theorem]{Lemma}

\newtheorem{proposition}[theorem]{Proposition}

\newenvironment{proof}[1][Proof]{\noindent{\textbf {#1}  }}  {\hfill$\Box$\bigskip}

\begin{document}

\title{Cut-norms and spectra of matrices}
\author{Vladimir Nikiforov\\{\small Department of Mathematical Sciences, University of Memphis, Memphis TN
38152}}
\maketitle

\begin{abstract}
One of the aims of this paper is to solve an open problem of Lov\'{a}sz about
relations between graph spectra and cut-distance. The paper starts with
several inequalities between two versions of the cut-norm and the two largest
singular values of arbitrary complex matrices, exteding, in particular, the
well-known graph-theoretical Expander Mixing Lemma and giving a hitherto
unknown converse of it.

Next, cut-distance is defined for Hermitian matrices, and, separately, for
arbitrary complex matrices; using these extensions, we give upper bounds on
the difference of corresponding eigenvalues and singular values of two
matrices, thus solving the problem of Lov\'{a}sz.

Finally, we deduce a spectral sampling theorem, which informally states that
almost all principal submatrices of a real symmetric matrix are spectrally
similar to it.\medskip


\end{abstract}

\section{Introduction}

In 1997, Frieze and Kannan \cite{FrKa99} introduced and studied the cut-norm
of matrices; ever since then this parameter kept getting new attention. It has
been extended and used for multidimensional matrices in \cite{AVKK03} and its
algorithmic aspects have been studied in \cite{AlNa04}. More recently,
starting with the cut-norm, Lov\'{a}sz and his coauthors in
\cite{LoSz06,LoSz07,BCLSV08,BCLSV09} defined a measure of similarity between
graphs, which they called the \emph{cut-distance} and used to investigate the
asymptotics of sequences of dense graphs.

It turned out that the cut-distance is related to many fundamental graph
parameters. In particular, in \cite{BCLSV09}, among many other things, it was
proved that if two graphs are close in cut-distance, then they are close
spectrally. Yet, since this specific result did not produce explicit
inequalities, Lov\'{a}sz \cite{Lov08} raised the problem to find the best
upper bound on the spectral difference in terms of the cut-distance of graphs.

One of the aims of this paper is to solve this problem and extend it to
arbitrary matrices. To this end, we start by establishing several tight
inequalities between two versions of the cut-norm and the two largest singular
values of arbitrary complex matrices. As first-hand applications of these
inequalities we extend the well-known graph-theoretical Expander Mixing Lemma
and its converse. In particular, we obtain a new converse of the Expander
Mixing Lemma, which is simpler than those in \cite{BiLi06,BoNi04,But06}.

Next, we extend the concept of cut-distance to Hermitian matrices, and,
separately, to arbitrary complex matrices; using these extensions, we give
upper bounds on the difference of corresponding eigenvalues and corresponding
singular values of two matrices.

As an application we deduce a spectral sampling theorem, which informally
states that almost all principal submatrices of a real symmetric matrix are
spectrally similar to it. These result complements results of \cite{RuVe07}
and \cite{ChLe08}.

The rest of the paper is organized as follows: in the remaining subsections of
the introduction we state our main results together with some discussions. All
proofs are collected in Section \ref{p}. At the end, some open question are raised.

\subsection{Notation and definitions}

First we introduce some notation and conventions. For undefined matrix
notation we refer the reader to \cite{HoJo85}. We write:\medskip

- $\mathcal{M}_{m,n}$ for the class of all complex matrices of size $m\times
n;$

- $\mathcal{H}_{n}$ for the class of all Hermitian matrices of size $n;$

- $\mathcal{P}_{n}$ for the class of all permutation matrices of size $n;$

- $J_{m,n}$ for the all ones matrix of size $m\times n,$ and set
$J_{n}=J_{n,n}.$

- $\left\langle \mathbf{x},\mathbf{y}\right\rangle $ for the standard inner
product in $\mathbb{C}^{n};$

- $\mathbf{y\otimes x}$ for the $m\times n$ matrix $\left[  y_{i}x_{j}\right]
$, where $\mathbf{x}=\left(  x_{1},\ldots,x_{n}\right)  $ and $\mathbf{y}%
=\left(  y_{1},\ldots,y_{m}\right)  .$\medskip

Given a matrix $A=\left[  a_{ij}\right]  \in\mathcal{M}_{m,n},$ we
write:\medskip

- $\left\vert A\right\vert _{\infty}$ for $\max_{i,j}\left\vert a_{ij}%
\right\vert ;$

- $\left\Vert A\right\Vert _{F}$ for the Frobenius norm $\sqrt{\sum
_{i,j}\left\vert a_{ij}\right\vert ^{2}};$

- $\left\Vert A\right\Vert _{2}$ for the operator norm of the linear map
$A:\mathbb{C}^{n}\rightarrow\mathbb{C}^{m}$;

- $\sigma_{1}\left(  A\right)  \geq\cdots\geq\sigma_{m}\left(  A\right)  $ for
the singular values of $A;$

- $\mu_{1}\left(  A\right)  \geq\cdots\geq\mu_{m}\left(  A\right)  $ for the
eigenvalues of $A$ if $A$ is Hermitian;

- $\Sigma\left(  A\right)  $ for the sum of the entries of $A;$

- $\rho\left(  A\right)  $ for $\Sigma\left(  A\right)  /mn;$

- $A^{\ast}$ for the conjugate transpose of $A;$

- $A\left[  X,Y\right]  $ for the submatrix of all $a_{ij}$ with $i\in X,$
$j\in Y,$ where $X\subset\left[  m\right]  ,$ $Y\subset\left[  n\right]
$.\bigskip

\begin{definition}
Following \cite{FrKa99}, for every $A\in\mathcal{M}_{m,n},$ define the
\emph{cut-norm} $\left\Vert A\right\Vert _{\square}$ of $A$ by%
\[
\left\Vert A\right\Vert _{\square}=\max_{X\subset\left[  m\right]  ,\text{
}Y\subset\left[  n\right]  }\frac{1}{mn}\left\vert \Sigma\left(  A\left[
X,Y\right]  \right)  \right\vert .
\]

\end{definition}

\begin{definition}
A similar, yet distinct norm $\left\Vert A\right\Vert _{\boxdot}$ can be
defined by
\[
\left\Vert A\right\Vert _{\boxdot}=\max_{X\subset\left[  m\right]  ,\text{
}Y\subset\left[  n\right]  ,\text{ }X,Y\neq\varnothing}\frac{1}{\sqrt
{\left\vert X\right\vert \left\vert Y\right\vert }}\left\vert \Sigma\left(
A\left[  X,Y\right]  \right)  \right\vert .
\]

\end{definition}

The norm $\left\Vert A\right\Vert _{\boxdot}$ is implicit in numerous papers
related to the second singular value and to expansion of graphs: specifically,
when $A$ is the adjacency matrix of a graph, the value $\left\Vert
A-\rho\left(  A\right)  J_{n}\right\Vert _{\boxdot}$ appeared first as
the\emph{ }$\alpha$\emph{-parameter }in Thomason \cite{Tho1,Tho2}; it was
developed further by Alon, Chung and Spencer \cite{AlCh88,AlSp92}, and more
recently it was studied in \cite{BiLi06,BoNi04,But06}.

Note that neither $\left\Vert A\right\Vert _{\square}$ nor $\left\Vert
A\right\Vert _{\boxdot}$ are sub-multiplicative; therefore, they are not
matrix norms in the strict sense as defined, say, in \cite{HoJo85}, Ch. 5.

\subsection{Bounds on cut-norms}

The following two upper bounds on $\sigma_{1}\left(  A\right)  $ in terms of
$\left\Vert A\right\Vert _{\square}$ and $\left\Vert A\right\Vert _{\boxdot}$
are the cornerstones of our investigation.

\begin{theorem}
\label{th0}Let $A\in\mathcal{M}_{m,n}.$ If $A$ is real, then%
\begin{equation}
\sigma_{1}\left(  A\right)  \leq2\sqrt{\left\vert A\right\vert _{\infty
}\left\Vert A\right\Vert _{\square}mn}; \label{gin1}%
\end{equation}
if $A$ is complex, then
\begin{equation}
\sigma_{1}\left(  A\right)  \leq4\sqrt{\left\vert A\right\vert _{\infty
}\left\Vert A\right\Vert _{\square}mn} \label{gin1.1}%
\end{equation}
and%
\begin{equation}
\sigma_{1}\left(  A\right)  \leq C\left\Vert A\right\Vert _{\boxdot}\sqrt{\log
m\log n} \label{gin2}%
\end{equation}
for some positive $C<10^{5.}.$ Inequalities (\ref{gin1}), (\ref{gin1.1}) and
(\ref{gin2}) are tight up to constant factors.
\end{theorem}

Inequalities (\ref{gin1}) and (\ref{gin2}) can be inverted to some extent.
Indeed, Schur's identity $\sigma_{1}\left(  A\right)  =\left\Vert A\right\Vert
_{2}$ (\cite{Sch12}) implies that
\begin{equation}
\sigma_{1}\left(  A\right)  \geq\left\Vert A\right\Vert _{\boxdot},
\label{gin3}%
\end{equation}
and consequently, $\sigma_{1}\left(  A\right)  \geq\left\Vert A\right\Vert
_{\square}\sqrt{mn}.$

In turn, inequality (\ref{gin3}) implies an extension of the Expander Mixing
Lemma, including its bipartite version which is implicit in Gowers
\cite{Gow08}, Lemma 2.9. For convenience, we restate this graph-theoretical
result:\medskip

\emph{Let }$G$ \emph{be a bipartite graph with vertex classes }$U$ \emph{and
}$V,$ \emph{and let }$A$\emph{\ be its biadjacency matrix. Suppose that }$G$
\emph{is semiregular, i.e., vertices belonging to the same vertex class have
the same degree.} \emph{Then }%
\[
\sigma_{2}\left(  A\right)  \geq\max_{S\subset U,\text{ }R\subset V,\text{
}S,R\neq\varnothing}\frac{1}{\sqrt{\left\vert S\right\vert \left\vert
R\right\vert }}\left\vert e\left(  S,R\right)  -\frac{e\left(  U,V\right)
}{\left\vert U\right\vert \left\vert V\right\vert }\left\vert S\right\vert
\left\vert R\right\vert \right\vert .
\]
Here $e\left(  X,Y\right)  $ stands for the number of edges $uv$ such that
$u\in X,v\in Y.$ The theorem below extends the Expander Mixing Lemma to any
matrices. Note that for nonnegative matrices essentially the same result has
been obtained by Butler in \cite{But06}, Theorem 1.

For $\mathbf{x}=\left(  x_{1},\ldots,x_{n}\right)  \in\mathbb{C}^{n},$
$\mathbf{y}=\left(  y_{1},\ldots,y_{m}\right)  \in\mathbb{C}^{m},$ write
$\mathbf{y\otimes x}$ for the $m\times n$ matrix $\left[  y_{i}x_{j}\right]
.$

\begin{theorem}
\label{EML}Let $A\in\mathcal{M}_{m,n}$ and let $\mathbf{x}\in\mathbb{C}^{n},$
$\mathbf{y}\in\mathbb{C}^{m}$ be two unit vectors such that $\sigma_{1}\left(
A\right)  =\left\langle A\mathbf{x},\mathbf{y}\right\rangle .$ Then
\[
\sigma_{2}\left(  A\right)  \geq\left\Vert A-\sigma_{1}\left(  A\right)
\overline{\mathbf{y}}\mathbf{\otimes x}\right\Vert _{\boxdot}.
\]
In particular, if $A$ is a nonnegative matrix with equal row sums and equal
column sums, then
\[
\sigma_{2}\left(  A\right)  \geq\left\Vert A-\rho\left(  A\right)
J_{m,n}\right\Vert _{\boxdot}.
\]

\end{theorem}

Along this line, inequalities (\ref{gin1}) and (\ref{gin2}) imply the
following upper bounds on $\sigma_{2}.$

\begin{theorem}
\label{CEML}For every $A\in\mathcal{M}_{m,n}$ we have
\begin{equation}
\sigma_{2}\left(  A\right)  \leq4\sqrt{\left\Vert A-\rho\left(  A\right)
J_{m,n}\right\Vert _{\square}mn} \label{sin1}%
\end{equation}
and%
\begin{equation}
\sigma_{2}\left(  A\right)  <C\left\Vert A-\rho\left(  A\right)
J_{m,n}\right\Vert _{\boxdot}\sqrt{\log m\log n} \label{sin2}%
\end{equation}
for some positive $C<10^{5}.$ Inequalities (\ref{sin1}) and (\ref{sin2}) are
tight up to constant factors.
\end{theorem}

In the above general matrix setup, inequality (\ref{sin2}) is new, but for
Hermitian matrices it is known from \cite{BoNi04}. For regular graphs somewhat
better results were obtained by Bilu and Linial \cite{BiLi06}, and for
nonnegative matrices, by Butler \cite{But06}. On the other hand, inequality
(\ref{sin1}) is entirely new; while it seems less subtle than (\ref{sin2}), it
is much easier to use.

\subsection{Extending the cut-distance to complex matrices}

Following the general idea of cut-distance for graphs, we shall define
cut-distance for arbitrary matrices. Note that, in fact, Lov\'{a}sz and his
coauthors have defined the cut-distance for real measurable functions
$f:\left[  0,1\right]  ^{2}\rightarrow\mathbb{R,}$ in particular for real
symmetric matrices. For Hermitian matrices we follow their footprints, but for
arbitrary complex matrices we make a necessary adjustment, producing in fact a
slightly different version of the cut-distance, even for graphs.

\subsubsection*{The cut-distance of Hermitian matrices}

Given $A=\left[  a_{ij}\right]  \in\mathcal{H}_{n}$ and integer $p\geq1,$ let
\[
A^{\left(  p\right)  }=A\otimes J_{p},
\]
where $\otimes$ denotes the Kronecker product. Thus, $A^{\left(  p\right)  }$
is obtained by replacing each entry $a_{ij}$ with the matrix $a_{ij}J_{p}.$
Note that $A^{\left(  p\right)  }\in\mathcal{H}_{np}$. Now, for every
$A,B\in\mathcal{H}_{n},$ define$\ \widehat{\delta}_{\square}\left(
A,B\right)  $ as%
\[
\widehat{\delta}_{\square}\left(  A,B\right)  =\min\left\{  \left\Vert
A-PBP^{-1}\right\Vert _{\square}:P\in\mathcal{P}_{n}\right\}  .
\]
Finally, extend the function $\widehat{\delta}_{\square}\left(  A,B\right)  $
to matrices of different sizes as follows: for every $A\in\mathcal{H}_{n}$ and
$B\in\mathcal{H}_{m},$ define the cut-distance $\delta_{\square}\left(
A,B\right)  $ as%
\[
\delta_{\square}\left(  A,B\right)  =\lim_{k\rightarrow\infty}\widehat{\delta
}_{\square}\left(  A^{\left(  km\right)  },B^{\left(  kn\right)  }\right)  .
\]
It is not immediate, but is rather simple to see that the limit above exists,
and moreover,
\[
\delta_{\square}\left(  A,B\right)  =\inf_{k}\widehat{\delta}_{\square}\left(
A^{\left(  km\right)  },B^{\left(  kn\right)  }\right)  .
\]
Note also that the function $\delta_{\square}\left(  A,B\right)  $ is
symmetric and satisfies the triangle inequality%
\[
\delta_{\square}\left(  A,B\right)  \leq\delta_{\square}\left(  A,C\right)
+\delta_{\square}\left(  C,B\right)
\]
for all Hermitian matrices $A,B,C.$ However, $\delta_{\square}\left(
A^{\left(  p\right)  },A^{\left(  q\right)  }\right)  =0,$ and so,
$\delta_{\square}\left(  \cdot,\cdot\right)  $ is not a true metric, but only
a pre-metric.

\subsubsection*{The cut-distance of arbitrary matrices}

The matrix setup allows an easy modification of $\delta_{\square}\left(
\cdot,\cdot\right)  $ for arbitrary complex matrices. Given $A=\left[
a_{ij}\right]  \in\mathcal{M}_{m,n}$ and two positive integers $p,q,$ let
\[
A^{\left(  p,q\right)  }=A\otimes J_{p,q}.
\]
Note that $A^{\left(  p,q\right)  }\in\mathcal{M}_{mp,nq}$. Now, for every
$A,B$ $\in\mathcal{M}_{m,n},$ define$\ \widehat{\delta}_{\boxminus}\left(
A,B\right)  $ as%
\[
\widehat{\delta}_{\boxminus}\left(  A,B\right)  =\min\left\{  \left\Vert
A-PBQ\right\Vert _{\square}:P\in\mathcal{P}_{m},\text{ }Q\in\mathcal{P}%
_{n}\right\}
\]
Finally, extend the function $\widehat{\delta}_{\square}\left(  A,B\right)  $
to matrices of different sizes as follows: for every $A\in\mathcal{M}_{m,n}$
and $B\in\mathcal{M}_{r,s},$ define $\delta_{\boxminus}\left(  A,B\right)  $
as%
\[
\delta_{\boxminus}\left(  A,B\right)  =\lim_{k\rightarrow\infty}%
\widehat{\delta}_{\boxminus}\left(  A^{\left(  kr,ks\right)  },B^{\left(
km,kn\right)  }\right)  .
\]
As in the case of Hermitian matrices, the above limit exists and we have
\[
\delta_{\boxminus}\left(  A,B\right)  =\inf_{k}\widehat{\delta}_{\boxminus
}\left(  A^{\left(  kr,ks\right)  },B^{\left(  km,kn\right)  }\right)  .
\]
Also, the function $\delta_{\boxminus}\left(  \cdot,\cdot\right)  $ is
symmetric and satisfies the triangle inequality, but is only a pre-metric.

Note that now, for Hermitian matrices we have two cut-distances:
$\delta_{\square}\left(  \cdot,\cdot\right)  $ and $\delta_{\boxminus}\left(
\cdot,\cdot\right)  .$ It is not difficult to prove that%
\[
\delta_{\boxminus}\left(  A,B\right)  \leq\delta_{\square}\left(  A,B\right)
\leq2\delta_{\boxminus}\left(  A,B\right)
\]
for every two Hermitian matrices $A$ and $B$.

\subsection{The spectral difference of matrices}

Having inequality (\ref{gin1}) and the definition of $\delta_{\square}\left(
\cdot,\cdot\right)  $ in hand, we can bound the difference of corresponding
eigenvalues of two Hermitian matrices $A$ and $B$ in terms of $\delta
_{\square}\left(  A,B\right)  .$ The main difficulties here come from the fact
that $A$ and $B$ can be of different size and consequently have a different
number of eigenvalues. But even when $A$ and $B$ are of the same size, there
may be complications due to a huge difference in the number of their positive
eigenvalues. Thus, the theorem below gives two conclusions from the same
premise: one when eigenvalue signs are taken into account (clauses ii.a and
ii.b), and one when they are not (clause i ).

\begin{theorem}
\label{th2}Let $n\geq m\geq1,$ and let $A\in\mathcal{H}_{n},$ $B\in
\mathcal{H}_{m}$ satisfy $\left\vert A\right\vert _{\infty}=\left\vert
B\right\vert _{\infty}=1.$ Then

(i) for every $i=1,\ldots,\left\lceil m/2\right\rceil ,$ we have%
\begin{align*}
\left\vert \frac{\mu_{i}\left(  A\right)  }{n}-\frac{\mu_{i}\left(  B\right)
}{m}\right\vert  &  \leq\frac{1}{\sqrt{n/2}}+\frac{1}{\sqrt{m/2}}%
+6\delta_{\square}\left(  A,B\right)  ^{1/2}\\
\left\vert \frac{\mu_{n-i+1}\left(  A\right)  }{n}-\frac{\mu_{m-i+1}\left(
B\right)  }{m}\right\vert  &  \leq\frac{1}{\sqrt{n/2}}+\frac{1}{\sqrt{m/2}%
}+6\delta_{\square}\left(  A,B\right)  ^{1/2}%
\end{align*}

(ii.a) if $\mu_{i}\left(  A\right)  \geq0$ and $\mu_{i}\left(  B\right)
\geq0,$ then$,$%
\[
\left\vert \frac{\mu_{i}\left(  A\right)  }{n}-\frac{\mu_{i}\left(  B\right)
}{m}\right\vert \leq6\delta_{\square}\left(  A,B\right)  ^{1/2}%
\]

(ii.b) if $\mu_{i}\left(  A\right)  \leq0$ and $\mu_{i}\left(  B\right)
\leq0,$ then
\[
\left\vert \frac{\mu_{n-i+1}\left(  A\right)  }{n}-\frac{\mu_{m-i+1}\left(
B\right)  }{m}\right\vert \leq6\delta_{\square}\left(  A,B\right)  ^{1/2}.
\]

\end{theorem}

Since in clause (i) of the above theorem eigenvalue signs are not taken into
account, the undesired term $\left(  n/2\right)  ^{-1/2}+\left(  m/2\right)
^{-1/2}$ appears in the right-hand side. In general, this term seems
unavoidable: indeed, taking $A=B^{\left(  k\right)  }$, we have $\delta
_{\square}\left(  A,B\right)  =0,$ but the difference
\[
\left\vert \frac{\mu_{i}\left(  B^{\left(  k\right)  }\right)  }{mk}-\frac
{\mu_{i}\left(  B\right)  }{m}\right\vert
\]
can be as large as $m^{-1/2}/2,$ say when $B$ is a Paley graph of sufficiently
large order $m$ and $i=\left\lceil m/2\right\rceil +1.$

Fortunately, for singular values, everything goes smoothly.

\begin{theorem}
\label{th3}Let $A\in\mathcal{M}_{m,n}$ and $B\in\mathcal{M}_{r,s}$ satisfy
$\left\vert A\right\vert _{\infty}=\left\vert B\right\vert _{\infty}=1.$ Then
for every $i=1,\ldots,\min\left(  m,n,r,s\right)  ,$ we have%
\[
\left\vert \frac{\sigma_{i}\left(  A\right)  }{\sqrt{mn}}-\frac{\sigma
_{i}\left(  B\right)  }{\sqrt{rs}}\right\vert \leq6\delta_{\boxminus}\left(
A,B\right)  ^{1/2}.
\]

\end{theorem}

\textbf{Remark. }If the matrices in Theorems \ref{th2} and \ref{th3} are real,
the coefficient $6$ in the right-hand side can be replaced by $3.$\textbf{ }

\subsection{Matrix sampling}

Alon, de la Vega, Kannan and Karpinski \cite{AVKK03} came up with a powerful
matrix sampling result, further improved by Borgs, Chayes, Lov\'{a}sz,
S\'{o}s, and Vesztergombi in \cite{BCLSV08}, Theorem 2.9; for convenience we
restate it in a slightly weaker form:\medskip

\emph{Let }$n\geq k\geq1$\emph{ and let }$A$\emph{ be a real symmetric matrix
of size }$n.$\emph{ Let }$B=A\left[  X,X\right]  ,$\emph{ where }$X$\emph{ is
a uniformly random subset of }$\left[  n\right]  $\emph{ of size }$k$\emph{.
Then }%
\[
\delta_{\square}\left(  A,B\right)  <10\left\vert A\right\vert _{\infty
}\left(  \log_{2}k\right)  ^{-1/2}.
\]
\emph{with probability at least }$1-\exp\left(  -k^{2}/\left(  2\log
_{2}k\right)  \right)  .$\medskip

In view of this theorem, we can use Theorems \ref{th2} to derive a spectral
sampling theorem for real symmetric matrices. There is a rich literature
dedicated to this topic, see, e.g., the references of \cite{RuVe07}; we shall
mention only two recent milestones: Chatterjee and Ledoux \cite{ChLe08} proved
that almost all principal submatrices of a Hermitian matrix $A$ have empirical
eigenvalue distribution close to the expected eigenvalue distribution. Prior
to that, Rudelson and Vershynin \cite{RuVe07} have obtained more precise
results, but only for the singular values of special submatrices. Here we take
an intermediate approach. We prove a sampling result about principal
submatrices of real symmetric matrices, bounding all eigenvalues of the sample
submatrix, but not attempting the level of precision as in \cite{RuVe07}. In
addition, our methods are much simpler than the methods of \cite{ChLe08} and
\cite{RuVe07}.

\begin{theorem}
\label{ssamp}Let $n\geq k\geq1$ and $A$ be a real symmetric matrix of size
$n.$ Let $B=A\left[  X,X\right]  ,$ where $X$ is a uniformly random subset of
$\left[  n\right]  $ of size $k$. Then with probability at least
$1-\exp\left(  -k^{2}/\left(  2\log_{2}k\right)  \right)  ,$ for every
$i=1,\ldots,k,$ we have

(i) if $\mu_{i}\left(  B\right)  \geq0,$ then%
\[
\left\vert \frac{\mu_{i}\left(  A\right)  }{n}-\frac{\mu_{i}\left(  B\right)
}{k}\right\vert <30\left(  \log_{2}k\right)  ^{-1/4};
\]

(ii) if $\mu_{i}\left(  B\right)  \leq0,$ then
\[
\left\vert \frac{\mu_{n-k+i}\left(  A\right)  }{n}-\frac{\mu_{i}\left(
B\right)  }{k}\right\vert <30\left(  \log_{2}k\right)  ^{-1/4}.
\]

\end{theorem}

\section{\label{p}Proofs}

\subsection{Proof of Theorem \ref{th0}}

For the proof of inequality (\ref{gin1}) we need a standard lemma that can be
traced back to \cite{FrKa99}. We prove it here for convenience.

\begin{lemma}
\label{le1}Let $A\in\mathcal{M}_{m,n},$ $\mathbf{x}\in\mathbb{R}^{n},$
$\mathbf{y}\in\mathbb{R}^{m}.$ Then%
\begin{equation}
\left\vert \left\langle A\mathbf{x},\mathbf{y}\right\rangle \right\vert
\leq4\left\Vert \mathbf{x}\right\Vert _{\infty}\left\Vert \mathbf{y}%
\right\Vert _{\infty}\left\Vert A\right\Vert _{\square}mn. \label{in1}%
\end{equation}
If $\mathbf{x}\in\mathbb{C}^{n},$ $\mathbf{y}\in\mathbb{C}^{m},$ then%
\begin{equation}
\left\vert \left\langle A\mathbf{x},\mathbf{y}\right\rangle \right\vert
\leq16\left\Vert \mathbf{x}\right\Vert _{\infty}\left\Vert \mathbf{y}%
\right\Vert _{\infty}\left\Vert A\right\Vert _{\square}mn. \label{in2}%
\end{equation}

\end{lemma}

\begin{proof}
Assume for simplicity that $\left\Vert \mathbf{x}\right\Vert _{\infty
}=\left\Vert \mathbf{y}\right\Vert _{\infty}=1.$ We shall prove first
(\ref{in1}). Since $\left\langle A\mathbf{u},\mathbf{v}\right\rangle $ maps
the cube $\left[  -1,1\right]  ^{m+n}$ linearly in each coordinate of
$\mathbf{u}$ and $\mathbf{v,}$ $\max\left\vert \left\langle A\mathbf{u}%
,\mathbf{v}\right\rangle \right\vert $ is attained for some $\mathbf{u}%
^{\prime}=\left(  u_{1}^{\prime},\ldots,u_{n}^{\prime}\right)  \in\left\{
-1,1\right\}  ^{n},$ and $\mathbf{v}^{\prime}=\left(  v_{1}^{\prime}%
,\ldots,v_{n}^{\prime}\right)  \in\left\{  -1,1\right\}  ^{m}.$ Set
\begin{align*}
R^{+}  &  =\left\{  x:v_{x}^{\prime}=1\right\}  ,\text{ \ }R^{-}=\left\{
x:v_{x}^{\prime}=-1\right\}  ,\\
C^{+}  &  =\left\{  x:u_{x}^{\prime}=1\right\}  ,\text{ \ }C^{-}=\left\{
x:u_{x}^{\prime}=-1\right\}  .
\end{align*}
Now we see that
\begin{align*}
\left\vert \left\langle A\mathbf{u}^{\prime},\mathbf{v}^{\prime}\right\rangle
\right\vert  &  =\left\vert \Sigma\left(  A\left[  R^{+},C^{+}\right]
\right)  +\Sigma\left(  A\left[  R^{-},C^{-}\right]  \right)  -\Sigma\left(
A\left[  R^{+},C^{-}\right]  \right)  -\Sigma\left(  A\left[  R^{-}%
,C^{+}\right]  \right)  \right\vert \\
&  \leq\left\vert \Sigma\left(  A\left[  R^{+},C^{+}\right]  \right)
\right\vert +\left\vert \Sigma\left(  A\left[  R^{-},C^{-}\right]  \right)
\right\vert +\left\vert \Sigma\left(  A\left[  R^{+},C^{-}\right]  \right)
\right\vert +\left\vert \Sigma\left(  A\left[  R^{-},C^{+}\right]  \right)
\right\vert \\
&  \leq4\left\Vert A\right\Vert _{\square}mn,
\end{align*}
completing the proof of (\ref{in1}).

To prove (\ref{in2}), suppose that $\mathbf{x}=\left(  x_{1},\ldots
,x_{n}\right)  ,$ $\mathbf{y}=\left(  y_{1},\ldots,y_{m}\right)  ,$ and set
\begin{align*}
\mathbf{x}_{0}  &  =\left(  \operatorname{Re}x_{1},\ldots,\operatorname{Re}%
x_{n}\right)  ,\text{ }\mathbf{x}_{1}=\left(  \operatorname{Im}x_{1}%
,\ldots,\operatorname{Im}x_{n}\right)  ,\text{ }\\
\mathbf{y}_{0}  &  =\left(  \operatorname{Re}y_{1},\ldots,\operatorname{Re}%
y_{m}\right)  ,\text{ }\mathbf{y}_{1}=\left(  \operatorname{Im}y_{1}%
,\ldots,\operatorname{Im}y_{m}\right)  .
\end{align*}
We have
\begin{align*}
\left\vert \left\langle A\mathbf{x},\mathbf{y}\right\rangle \right\vert  &
=\left\vert \left\langle A\mathbf{x}_{0},\mathbf{y}_{0}\right\rangle
-\left\langle A\mathbf{x}_{0},\mathbf{y}_{1}\right\rangle i+\left\langle
A\mathbf{x}_{1},\mathbf{y}_{0}\right\rangle i+\left\langle A\mathbf{x}%
_{1},\mathbf{y}_{1}\right\rangle \right\vert \\
&  \leq\left\vert \left\langle A\mathbf{x}_{0},\mathbf{y}_{0}\right\rangle
\right\vert +\left\vert \left\langle A\mathbf{x}_{0},\mathbf{y}_{1}%
\right\rangle \right\vert +\left\vert \left\langle A\mathbf{x}_{1}%
,\mathbf{y}_{0}\right\rangle \right\vert +\left\vert \left\langle
A\mathbf{x}_{1},\mathbf{y}_{1}\right\rangle \right\vert .
\end{align*}
Since $\mathbf{x}_{0},\mathbf{x}_{1},\mathbf{y}_{0},\mathbf{y}_{1}$ are real,
inequality (\ref{in2}) follows from (\ref{in1}).
\end{proof}

For the proof of (\ref{gin2}) we need the following lemma, proved in
\cite{BoNi04}.

\begin{lemma}
\label{lapp} Let $p\geq1$, $n\geq1$ and $0<\varepsilon<1$. Then for every
$\mathbf{x}=\left(  x_{1},\ldots,x_{n}\right)  \in\mathbb{C}^{n}$ with
$\left\Vert \mathbf{x}\right\Vert =1$, there is a vector $\mathbf{y}=\left(
y_{1},\ldots,y_{n}\right)  \in\mathbb{C}^{n}$ such that $y_{i}$ take no more
than
\[
\left\lceil \frac{8\pi}{\varepsilon}\right\rceil \left\lceil \frac
{4}{\varepsilon}\log\frac{4n}{\varepsilon}\right\rceil
\]
values and $\left\Vert \mathbf{x}-\mathbf{y}\right\Vert \leq\varepsilon$.
\end{lemma}

\begin{proof}
[\textbf{Proof of inequalities (\ref{gin1.1}) and (\ref{gin1})}]We shall prove
first (\ref{gin1.1}). Let $A=\left[  a_{ij}\right]  $, $A^{\ast}=\left[
a_{ij}^{\ast}\right]  ,$ $A^{\ast}A=\left[  b_{ij}\right]  .$ For every
$j\in\left[  n\right]  ,$ set%
\[
\mathbf{c}_{j}=\left(  a_{1j},a_{2j},\ldots,a_{mj}\right)  .
\]
Select $i\in\left[  n\right]  $ so that $\sum_{j\in\left[  n\right]
}\left\vert b_{ij}\right\vert $ is maximal. It is well-known that%
\[
\sigma_{1}^{2}\left(  A\right)  =\mu_{1}\left(  A^{\ast}A\right)  \leq
\sum_{j\in\left[  n\right]  }\left\vert b_{ij}\right\vert .
\]
Note also that
\[
b_{ij}=\sum_{k\in\left[  m\right]  }a_{ik}^{\ast}a_{kj}=\sum_{k\in\left[
m\right]  }a_{kj}\overline{a_{ki}}=\left\langle \mathbf{c}_{j},\mathbf{c}%
_{i}\right\rangle
\]
For every $j\in\left[  n\right]  ,$ set%
\[
x_{j}=\left\{
\begin{array}
[c]{cc}%
\left\vert \left\langle \mathbf{c}_{j},\mathbf{c}_{i}\right\rangle \right\vert
/\left\langle \mathbf{c}_{j},\mathbf{c}_{i}\right\rangle  & \text{if
}\left\langle \mathbf{c}_{j},\mathbf{c}_{i}\right\rangle \neq0\\
0 & \text{if }\left\langle \mathbf{c}_{j},\mathbf{c}_{i}\right\rangle =0
\end{array}
\right.
\]
and let $\mathbf{x}=\left(  x_{1},\ldots,x_{n}\right)  .$ Also for every
$k\in\left[  m\right]  ,$ set $y_{k}=a_{ki}$ and let $\mathbf{y}=\left(
y_{1},\ldots,y_{m}\right)  $. Note that
\[
\sum_{j\in\left[  n\right]  }\left\vert b_{ij}\right\vert =\sum_{j\in\left[
n\right]  }\left\langle \mathbf{c}_{j},\mathbf{c}_{i}\right\rangle x_{j}%
=\sum_{j\in\left[  n\right]  }\sum_{k\in\left[  m\right]  }\overline{a_{ki}%
}a_{kj}x_{j}=\sum_{k\in\left[  m\right]  }\sum_{j\in\left[  n\right]
}\overline{a_{ki}}a_{kj}x_{j}=\left\langle A\mathbf{x},\mathbf{y}\right\rangle
.
\]
Since $\left\vert \mathbf{y}\right\vert _{\infty}\leq1,$ in view of
(\ref{in2}), we see that
\[
\left\langle A\mathbf{x},\mathbf{y}\right\rangle \leq16\left\Vert A\right\Vert
_{\square}\left\vert \mathbf{x}\right\vert _{\infty}\left\vert \mathbf{y}%
\right\vert _{\infty}=16\left\Vert A\right\Vert _{\square}\left\vert
A\right\vert _{\infty},
\]
completing the proof of inequality (\ref{gin1.1}).

Inequality (\ref{gin1}) follows likewise, using (\ref{in1}) instead of
(\ref{in2}).

To prove that inequality (\ref{gin1}) is tight, define a square symmetric
matrix $A=\left[  a_{ij}\right]  $ of size $2n+1,$ by letting%
\[
a_{1i}=a_{i1}=\left\{
\begin{array}
[c]{ll}%
1 & \text{if }2\leq i\leq n+1,\\
-1 & \text{if }n+1<i\leq2n+1
\end{array}
\right.
\]
and let all other entries of $A$ to be $0.$ We easily find that $\sigma
_{1}\left(  A\right)  =\sqrt{2n}.$ Also $\left\Vert A\right\Vert _{\square
}\left(  2n+1\right)  ^{2}=2n,$ for if $X\subset\left[  2n+1\right]  ,$
$Y\subset\left[  2n+1\right]  $ are such that $\left\vert \Sigma\left(
A\left[  X,Y\right]  \right)  \right\vert $ is maximal, the only contributions
to $\left\vert \Sigma\left(  A\left[  X,Y\right]  \right)  \right\vert $ come
from the first row and the first column, and each of them can be at most $n.$
Thus, we have%
\[
\sigma_{1}^{2}\left(  A\right)  =2n=\left\vert A\right\vert _{\infty
}\left\Vert A\right\Vert _{\square}\left(  2n+1\right)  ^{2},
\]
and so inequality (\ref{gin1}) is tight up to a factor of $2$ and inequality
(\ref{gin1.1}) is tight up to a factor of $4$.
\end{proof}

\begin{proof}
[\textbf{Proof of inequality (\ref{gin2})}]By Schur's identity $\sigma
_{1}\left(  A\right)  =\left\Vert A\right\Vert _{2},$ there exists unit
vectors $\mathbf{x}=\left(  x_{1},\ldots,x_{n}\right)  \in\mathbb{C}^{n}$ and
$\mathbf{y}=\left(  y_{1},\ldots,y_{m}\right)  \in\mathbb{C}^{m}$ such that
\[
\sigma_{1}\left(  A\right)  =\left\langle A\mathbf{x},\mathbf{y}\right\rangle
.
\]
Applying Lemma \ref{lapp} with $\varepsilon=1/3,$ we can find vectors%
\[
\mathbf{x}^{\prime}=\left(  x_{1}^{\prime},\ldots,x_{n}^{\prime}\right)
=\text{ \ }\mathbf{y}^{\prime}=\left(  y_{1}^{\prime},\ldots,y_{m}^{\prime
}\right)
\]
such that $x_{i}^{\prime}$ take $p$ distinct values $\alpha_{1},\ldots
,\alpha_{p}$ and $y_{i}^{\prime}$ take $q$ distinct values $\beta_{1}%
,\ldots,\beta_{q},$ and
\begin{align}
\left\Vert \mathbf{x}-\mathbf{x}^{\prime}\right\Vert  &  <1/3,\nonumber\\
\left\Vert \mathbf{y}-\mathbf{y}^{\prime}\right\Vert  &  <1/3,\label{pb}\\
p  &  \leq12\left\lceil 24\pi\right\rceil \left\lceil \log12n\right\rceil
=912\left\lceil \log12n\right\rceil ,\label{qb}\\
q  &  \leq12\left\lceil 24\pi\right\rceil \left\lceil \log12m\right\rceil
=912\left\lceil \log12m\right\rceil .\nonumber
\end{align}
For every $i\in\left[  p\right]  ,$ $j\in\left[  q\right]  ,$ let%
\begin{align*}
N_{i}  &  =\left\{  u:x_{u}^{\prime}=\alpha_{i}\right\}  ,\\
M_{j}  &  =\left\{  u:y_{u}^{\prime}=\beta_{j}\right\}  .
\end{align*}
Clearly, $N_{1}\cup\cdots\cup N_{p}$ and $M_{1}\cup\cdots\cup M_{q}$ are
partitions of $\left[  n\right]  $ and $\left[  m\right]  $.

Our first goal is to prove that
\begin{equation}
\sigma_{1}\left(  A\right)  \leq\frac{9}{2}\left\vert \left\langle
A\mathbf{x}^{\prime},\mathbf{y}^{\prime}\right\rangle \right\vert .
\label{eq4}%
\end{equation}
Indeed, we have
\begin{align*}
\left\vert \left\langle A\mathbf{x},\mathbf{y}\right\rangle -\left\langle
A\mathbf{x}^{\prime},\mathbf{y}^{\prime}\right\rangle \right\vert  &
=\left\vert \left\langle A\mathbf{x},\mathbf{y}\right\rangle -\left\langle
A\mathbf{x}^{\prime},\mathbf{y}\right\rangle +\left\langle A\mathbf{x}%
^{\prime},\mathbf{y}\right\rangle -\left\langle A\mathbf{x}^{\prime
},\mathbf{y}^{\prime}\right\rangle \right\vert \\
&  \leq\left\vert \left\langle A\left(  \mathbf{x-x}^{\prime}\right)
,\mathbf{y}\right\rangle \right\vert +\left\vert \left\langle A\mathbf{x}%
^{\prime}\mathbf{,y-y}^{\prime}\right\rangle \right\vert \\
&  \leq\sigma_{1}\left(  A\right)  \left(  \left\Vert \mathbf{x-x}^{\prime
}\right\Vert \left\Vert \mathbf{y}\right\Vert +\left\Vert \mathbf{x}^{\prime
}\right\Vert \left\Vert \mathbf{y-y}^{\prime}\right\Vert \right) \\
&  \leq\sigma_{1}\left(  A\right)  \left(  \left\Vert \mathbf{x-x}^{\prime
}\right\Vert \left\Vert \mathbf{y}\right\Vert +\left(  \left\Vert
\mathbf{x}\right\Vert +\left\Vert \mathbf{x-x}^{\prime}\right\Vert \right)
\left\Vert \mathbf{y-y}^{\prime}\right\Vert \right) \\
&  \leq\sigma_{1}\left(  A\right)  \left(  2+\frac{1}{3}\right)  \frac{1}%
{3}\leq\frac{7}{9}\sigma_{1}\left(  A\right)  ,
\end{align*}
implying that%
\[
\left\vert \sigma_{1}\left(  A\right)  \right\vert -\left\vert \left\langle
A\mathbf{x}^{\prime}\mathbf{,y}^{\prime}\right\rangle \right\vert \leq\frac
{7}{9}\sigma_{1}\left(  A\right)  .
\]
and inequality (\ref{eq4}) follows.

Now, define the matrix $C=\left[  c_{ij}\right]  \in\mathcal{M}_{p,q}$ by
\[
c_{ij}=\frac{1}{\sqrt{\left\vert M_{i}\right\vert \left\vert N_{j}\right\vert
}}\sum_{u\in M_{i}}\sum_{v\in N_{j}}a_{uv}.
\]
For every $i\in\left[  p\right]  ,$ $j\in\left[  q\right]  ,$ set $s_{i}%
=\sqrt{\left\vert N_{i}\right\vert }\alpha_{i},$ $t_{j}=\sqrt{\left\vert
M_{j}\right\vert }\beta_{j}$ and let $\mathbf{t}=\left(  t_{1},...,t_{q}%
\right)  ,$ $\mathbf{s}=\left(  s_{1},...,s_{p}\right)  .$ Clearly $\left\Vert
\mathbf{s}\right\Vert =\left\Vert \mathbf{x}^{\prime}\right\Vert $ and
$\left\Vert \mathbf{t}\right\Vert =\left\Vert \mathbf{y}^{\prime}\right\Vert
.$ Also, we see that
\begin{align*}
\left\vert \left\langle A\mathbf{x}^{\prime}\mathbf{,y}^{\prime}\right\rangle
\right\vert  &  =\left\vert \sum_{i=1}^{m}\sum_{j=1}^{n}a_{ij}x_{j}^{\prime
}\overline{y_{i}^{\prime}}\right\vert =\left\vert \sum_{j=1}^{p}\sum_{i=1}%
^{q}s_{j}\overline{t_{i}}\frac{1}{\sqrt{\left\vert M_{i}\right\vert \left\vert
N_{j}\right\vert }}\sum_{u\in M_{i}}\sum_{v\in Nj}a_{uv}\right\vert \\
&  =\left\vert \sum_{j=1}^{p}\sum_{i=1}^{q}c_{ij}s_{j}\overline{t_{i}%
}\right\vert \leq\sigma_{1}\left(  C\right)  \left\Vert \mathbf{s}\right\Vert
\left\Vert \mathbf{t}\right\Vert =\sigma_{1}\left(  C\right)  \left\Vert
\mathbf{x}^{\prime}\right\Vert \left\Vert \mathbf{y}^{\prime}\right\Vert \\
&  \leq\sigma_{1}\left(  C\right)  \left(  \left\Vert \mathbf{x}\right\Vert
+\frac{1}{3}\right)  \left(  \left\Vert \mathbf{y}\right\Vert +\frac{1}%
{3}\right)  =\frac{16}{9}\sigma_{1}\left(  C\right)
\end{align*}
Hence, in view of (\ref{eq4}), (\ref{pb}) and (\ref{qb}), we see that
\begin{align*}
\sigma_{1}\left(  A\right)   &  \leq8\sigma_{1}\left(  C\right)
\leq\left\Vert C\right\Vert _{F}\leq8\sqrt{pq}\max_{i,j\in\left[  m\right]
}\left\vert c_{ij}\right\vert \\
&  \leq8\cdot912\sqrt{\log12n\log12m}\left\Vert A\right\Vert _{\boxdot}.
\end{align*}
To complete the proof of (\ref{gin2}) assume that $n\geq2$ and $m\geq2$ and
observe that
\begin{align*}
8\cdot912\sqrt{\log12n\log12m}  &  \leq8\cdot912\sqrt{\left(  10\ln
12+1\right)  }\sqrt{\log n\log m}\\
&  <10^{5}\sqrt{\log n\log m}.
\end{align*}

To prove that inequality (\ref{gin2}) is tight up to a constant factor, define
a square symmetric matrix $A=\left[  a_{ij}\right]  $ of size $n$ by letting
$a_{ij}=\left(  ij\right)  ^{-1/2}.$ Set $s_{n}=\sum_{i=1}^{n}1/i.$ It is easy
to see that the vector $\mathbf{x}=\left(  x_{1},\ldots,x_{n}\right)  ,$ where
$x_{i}=\left(  is_{n}\right)  ^{-1/2}$ is of length $1$ and thus satisfies,
\[
\sigma_{1}\left(  A\right)  \geq\left\langle A\mathbf{x},\mathbf{x}%
\right\rangle =\sum_{i=1}^{n}\sum_{j=1}^{n}\frac{1}{ijs_{n}}=s_{n}>\log n.
\]
On the other hand, let $X\subset\left[  n\right]  ,$ $Y\subset\left[
n\right]  $ be such that
\[
\left\Vert A\right\Vert _{\boxdot}=\frac{1}{\sqrt{\left\vert X\right\vert
\left\vert Y\right\vert }}\left\vert \Sigma\left(  A\left[  X,Y\right]
\right)  \right\vert
\]
Clearly $X=\left[  p\right]  ,$ $Y=\left[  q\right]  $ for some $p,q\in\left[
n\right]  $. We thus have
\[
\frac{1}{\sqrt{\left\vert X\right\vert \left\vert Y\right\vert }}\left\vert
\Sigma\left(  A\left[  X,Y\right]  \right)  \right\vert =\frac{1}{\sqrt{pq}%
}\sum_{i=1}^{p}\sum_{j=1}^{q}\frac{1}{\sqrt{ij}}<\frac{4\sqrt{pq}}{\sqrt{pq}%
}=4.
\]
Hence,
\[
\sigma_{1}\left(  A\right)  >\log n>\frac{1}{4}4\log n>\frac{1}{4}\left\Vert
A\right\Vert _{\boxdot}\log n,
\]
and so, inequality (\ref{gin2}) is tight up to a constant factor.
\end{proof}

\subsection{Proofs of Theorems \ref{EML} and \ref{CEML}}

\begin{proof}
[\textbf{Proof of Theorem \ref{EML}}]The proof is essentially a tautology of
the singular value decomposition theorem (see, e.g., \cite{HoJo85}, Ch. 7).
Let
\[
A=\sigma_{1}\left(  A\right)  \overline{\mathbf{y}}\otimes\mathbf{x+}%
\sum_{i=2}^{m}\sigma_{i}\left(  A\right)  \overline{\mathbf{y}_{i}}%
\otimes\mathbf{x}_{i}%
\]
be a singular value decomposition of $A,$ where $\mathbf{y},\mathbf{y}%
_{2},\ldots,\mathbf{y}_{m}\in\mathbb{C}^{m}$ are unit orthogonal left singular
vectors and $\mathbf{x},\mathbf{x}_{2},\ldots,\mathbf{x}_{m}\in\mathbb{C}^{n}$
are unit orthogonal right singular vectors to $\sigma_{1}\left(  A\right)
,\sigma_{2}\left(  A\right)  ,\ldots,\sigma_{m}\left(  A\right)  .$ Hence,
\[
A-\sigma_{1}\left(  A\right)  \overline{\mathbf{y}}\otimes\mathbf{x}%
=\sum_{i=2}^{m}\sigma_{i}\left(  A\right)  \overline{\mathbf{y}_{i}}%
\otimes\mathbf{x}_{i},
\]
and so, $\sigma_{2}\left(  A\right)  =\sigma_{1}\left(  A-\sigma_{1}\left(
A\right)  \overline{\mathbf{y}}\otimes\mathbf{x}\right)  .$ Now (\ref{gin3})
implies that
\[
\sigma_{2}\left(  A\right)  =\sigma_{1}\left(  A-\sigma_{1}\left(  A\right)
\overline{\mathbf{y}}\otimes\mathbf{x}\right)  \geq\left\Vert A-\sigma
_{1}\left(  A\right)  \mathbf{x\otimes y}\right\Vert _{\boxdot}.
\]
If $A$ is nonnegative and its row sums are equal and its column sums are
equal, then we can choose
\[
\mathbf{x}=\left(  \frac{1}{\sqrt{n}},\ldots,\frac{1}{\sqrt{n}}\right)
,\text{ \ \ }\mathbf{y}=\left(  \frac{1}{\sqrt{m}},\ldots,\frac{1}{\sqrt{m}%
}\right)
\]
and so,
\[
\sigma_{1}\left(  A\right)  \overline{\mathbf{y}}\otimes\mathbf{x}=\rho\left(
A\right)  J_{m,n},
\]
completing the proof.
\end{proof}

\begin{proof}
[\textbf{Proof of Theorem \ref{CEML}}]Weyl's inequalities for singular values
(see, e.g., \cite{HoJo94}, Theorem 3.3.16) state that
\[
\sigma_{2}\left(  X+Y\right)  \leq\sigma_{1}\left(  X\right)  +\sigma
_{2}\left(  Y\right)  .
\]
Setting $Y=\rho\left(  A\right)  J_{m,n},$ $X=A-Y,$ and noting that
$\sigma_{2}\left(  Y\right)  =0,$ inequalities (\ref{sin1}) and (\ref{sin2})
follow from (\ref{gin1}) and (\ref{gin2}) respectively.

To see that inequality (\ref{sin2}) is tight, define the square symmetric
matrix $B=\left[  b_{ij}\right]  $ of size $n$ by letting $b_{ij}=\left(
ij\right)  ^{-1/2},$ and set
\[
A=\left(
\begin{array}
[c]{cc}%
J_{n}+B & J_{n}-B\\
J_{n}-B & J_{n}+B
\end{array}
\right)  .
\]
Clearly $\rho\left(  A\right)  =1,$ and so
\[
A-\rho\left(  A\right)  J_{2n}=\left(
\begin{array}
[c]{cc}%
B & -B\\
-B & B
\end{array}
\right)
\]
As shown in the proof of Theorem \ref{th0}, $\left\Vert B\right\Vert
_{\boxdot}<4,$ and so
\[
\left\Vert A-\rho\left(  A\right)  J_{2n}\right\Vert _{\boxdot}\leq4\left\Vert
B\right\Vert _{\boxdot}<16.
\]
On the other hand, $\sigma_{1}\left(  A\right)  =\mu_{1}\left(  A\right)
=2n,$ and the all ones vector is an eigenvector to $\mu_{1}\left(  A\right)
.$ It is easy to check that the $2n$-vector
\[
\left(  1,2^{-1/2},\ldots,n^{-1/2},-1,-2^{-1/2},\ldots,-n^{-1/2}\right)  ,
\]
is an eigenvector of $A$ to the eigenvalue $2\sum_{i}^{n}1/i.$ Hence,
\[
\sigma_{2}\left(  A\right)  \geq2\sum_{i}^{n}1/i>2\log n\geq\frac{1}%
{8}\left\Vert A-\rho\left(  A\right)  J_{2n}\right\Vert _{\boxdot}\sqrt{\log
n\log n}.
\]
and so (\ref{sin2}) is tight up to a constant factor.
\end{proof}

\subsection{Proof of Theorem \ref{th2}}

The following two facts are derived by straightforward methods.

\begin{proposition}
\label{pro1} Let $A\in\mathcal{H}_{n}$ and $k\geq2$. Then the eigenvalues of
$A^{\left(  k\right)  }$ are $k\mu_{1}\left(  A\right)  ,\ldots,k\mu
_{n}\left(  A\right)  $ together with $\left(  k-1\right)  n$ additional $0$'s.
\end{proposition}

\begin{proposition}
\label{pro2}Let $A\in\mathcal{M}_{m,n}$ and $p,q\geq2$. Then the singular
values of $A^{\left(  p,q\right)  }$ are
\[
\sqrt{pq}\sigma_{1}\left(  A\right)  ,\ldots,\sqrt{pq}\sigma_{m}\left(
A\right)
\]
and the rest are zeroes.
\end{proposition}

For the proof of Theorem \ref{th2} we shall show that the extremal $k$
eigenvalues of $A^{\left(  k\right)  }$ are roughly proportional to the
corresponding eigenvalues of $A.$

\begin{lemma}
\label{prop}Let $k\geq2.$ Then for every $i=1,\ldots,n,$
\begin{align}
0  &  \leq\frac{\mu_{i}\left(  A^{\left(  k\right)  }\right)  }{kn}-\frac
{\mu_{i}\left(  A\right)  }{n}\leq\frac{\left\Vert A\right\Vert _{F}}%
{n\sqrt{n-i+1}},\label{i1}\\
0  &  \geq\frac{\mu_{n-i+1}\left(  A^{\left(  k\right)  }\right)  }{tn}%
-\frac{\mu_{n-i+1}\left(  A\right)  }{n}\geq-\frac{\left\Vert A\right\Vert
_{F}}{n\sqrt{n-i+1}}, \label{i2}%
\end{align}

\end{lemma}

\begin{proof}
We shall prove only (\ref{i1}); inequality (\ref{i2}) follows likewise,
applying (\ref{i1}) to $-A.$ Note that Proposition \ref{pro2} implies that
$A^{\left(  k\right)  }$ and $A$ have the same number of positive eigenvalues.
If $\mu_{i}\left(  A^{\left(  k\right)  }\right)  >0,$ then $\mu_{i}\left(
A\right)  >0$ and $\mu_{i}\left(  A^{\left(  k\right)  }\right)  =k\mu
_{i}\left(  A\right)  ,$ so (\ref{i1}) holds. If $\mu_{i}\left(  A^{\left(
k\right)  }\right)  \leq0,$ then $\mu_{i}\left(  A\right)  \leq0$ and so
\[
0\geq\mu_{i}\left(  A\right)  \geq\cdots\geq\mu_{n}\left(  A\right)  .
\]
Hence, inequality (\ref{i1}) follows from
\[
\left(  n-i+1\right)  \mu_{i}^{2}\left(  A\right)  \leq\sum_{j=i}^{n}\mu
_{j}^{2}\left(  A\right)  \leq\left\Vert A\right\Vert _{F}^{2}.
\]

\end{proof}

\begin{proof}
[\textbf{Proof of Theorem \ref{th2}}]Let $k$ be a positive integer. By the
definition of $\widehat{\delta}_{\square}\left(  \cdot,\cdot\right)  ,$ there
is a permutation matrix $P\in\mathcal{P}_{mnk}$ such that
\[
\widehat{\delta}_{\square}\left(  A^{\left(  mk\right)  },B^{\left(
nk\right)  }\right)  =\left\Vert A^{\left(  mk\right)  }-PB^{\left(
nk\right)  }P^{-1}\right\Vert _{\square}.
\]
Referring to \cite{HoJo94}, Theorem 3.3.16, we have
\[
\left\vert \mu_{i}\left(  A^{\left(  mk\right)  }\right)  -\mu_{i}\left(
B^{\left(  nk\right)  }\right)  \right\vert =\left\vert \mu_{i}\left(
A^{\left(  mk\right)  }\right)  -\mu_{i}\left(  PB^{\left(  nk\right)  }%
P^{-1}\right)  \right\vert \leq\sigma_{1}\left(  A^{\left(  mk\right)
}-PB^{\left(  nk\right)  }P^{-1}\right)  .
\]
Now, inequality (\ref{gin1.1}) implies that
\begin{align*}
\left\vert \mu_{i}\left(  A^{\left(  mk\right)  }\right)  -\mu_{i}\left(
B^{\left(  nk\right)  }\right)  \right\vert  &  \leq\sigma_{1}\left(
A^{\left(  mk\right)  }-PB^{\left(  nk\right)  }P^{-1}\right) \\
&  \leq4\sqrt{2\left\Vert A^{\left(  mk\right)  }-PB^{\left(  nk\right)
}P^{-1}\right\Vert _{\square}}mnk\\
&  \leq6\sqrt{\widehat{\delta}_{\square}\left(  A^{\left(  mk\right)
},B^{\left(  nk\right)  }\right)  }mnk.
\end{align*}
and so%
\begin{equation}
\frac{1}{mnk}\left\vert \mu_{i}\left(  A^{\left(  mk\right)  }\right)
-\mu_{i}\left(  B^{\left(  nk\right)  }\right)  \right\vert \leq
6\sqrt{\widehat{\delta}_{\square}\left(  A^{\left(  mk\right)  },B^{\left(
nk\right)  }\right)  } \label{i3}%
\end{equation}

To prove (i), note that the triangle inequality and Lemma \ref{prop} imply
that
\begin{align*}
\left\vert \frac{\mu_{i}\left(  A\right)  }{n}-\frac{\mu_{i}\left(  B\right)
}{m}\right\vert  &  \leq\left\vert \frac{\mu_{i}\left(  A\right)  }{n}%
-\frac{\mu_{i}\left(  A^{\left(  mk\right)  }\right)  }{mnk}\right\vert
+\left\vert \frac{\mu_{i}\left(  B\right)  }{m}-\frac{\mu_{i}\left(
B^{\left(  nk\right)  }\right)  }{mnk}\right\vert +\left\vert \frac{\mu
_{i}\left(  A^{\left(  mk\right)  }\right)  }{mnk}-\frac{\mu_{i}\left(
B^{\left(  nk\right)  }\right)  }{mnk}\right\vert \\
&  \leq\frac{\left\Vert A\right\Vert _{F}}{n\sqrt{n-i+1}}+\frac{\left\Vert
B\right\Vert _{F}}{m\sqrt{m-i+1}}+6\sqrt{\widehat{\delta}_{\square}\left(
A^{\left(  mk\right)  },B^{\left(  nk\right)  }\right)  }\\
&  \leq\frac{1}{\sqrt{n-i+1}}+\frac{1}{\sqrt{m-i+1}}+6\sqrt{\widehat{\delta
}_{\square}\left(  A^{\left(  mk\right)  },B^{\left(  nk\right)  }\right)  }.
\end{align*}
Letting $k$ tend to infinity and passing to limits in the above inequality, we
obtain%
\[
\left\vert \frac{\mu_{i}\left(  A\right)  }{n}-\frac{\mu_{i}\left(  B\right)
}{m}\right\vert \leq\frac{1}{\sqrt{n-i+1}}+\frac{1}{\sqrt{m-i+1}}%
+6\sqrt{\delta_{\square}\left(  A,B\right)  }.
\]
Hence, for every $i=1,\ldots,\left\lceil m/2\right\rceil ,$
\begin{align*}
\left\vert \frac{\mu_{i}\left(  A\right)  }{n}-\frac{\mu_{i}\left(  B\right)
}{m}\right\vert  &  \leq\frac{1}{\sqrt{n/2}}+\frac{1}{\sqrt{m/2}}%
+6\sqrt{\delta_{\square}\left(  A,B\right)  },\\
\left\vert \frac{\mu_{n-i+1}\left(  A\right)  }{n}-\frac{\mu_{m-i+1}\left(
B\right)  }{m}\right\vert  &  \leq\frac{1}{\sqrt{n/2}}+\frac{1}{\sqrt{m/2}%
}+6\sqrt{\delta_{\square}\left(  A,B\right)  }.
\end{align*}

Now let us prove (ii.a). Suppose that $\mu_{i}\left(  A\right)  \geq0$ and
$\mu_{i}\left(  B\right)  \geq0.$ Then using Proposition \ref{pro1} and
(\ref{i3}), we find that
\[
\left\vert \frac{\mu_{i}\left(  A\right)  }{n}-\frac{\mu_{i}\left(  B\right)
}{m}\right\vert =\frac{1}{mnk}\left\vert \mu_{i}\left(  A^{\left(  mk\right)
}\right)  -\mu_{i}\left(  B^{\left(  nk\right)  }\right)  \right\vert
\leq6\sqrt{\widehat{\delta}_{\square}\left(  A^{\left(  mk\right)
},B^{\left(  nk\right)  }\right)  }.
\]
Letting $k$ tend to infinity and passing to limits (ii.a) follows. The clause
(ii.b) follows by a similar argument.
\end{proof}

\begin{proof}
[\textbf{Proof of Theorem \ref{th3}}]The proof is a straightforward
modification of the proof of Theorem \ref{th2}. Let $k$ be a positive integer.
By the definition of $\widehat{\delta}_{\boxminus}\left(  \cdot,\cdot\right)
,$ there exist permutation matrices $P\in\mathcal{P}_{mnk}$ and $Q\in
\mathcal{P}_{rsk}$ such that
\[
\widehat{\delta}_{\boxminus}\left(  A^{\left(  rk,sk\right)  },B^{\left(
mk,nk\right)  }\right)  =\left\Vert A^{\left(  rk,sk\right)  }-PB^{\left(
mk,nk\right)  }Q\right\Vert _{\square}.
\]
Since $\sigma_{i}\left(  B^{\left(  mk,nk\right)  }\right)  =\sigma_{i}\left(
PB^{\left(  mk,nk\right)  }Q\right)  ,$ referring to the inequality%
\[
\left\vert \sigma_{i}\left(  X\right)  -\sigma_{i}\left(  Y\right)
\right\vert \leq\sigma_{1}\left(  X-Y\right)  ,
\]
(see, e.g., \cite{HoJo94}, Theorem 3.3.16,), we obtain
\[
\left\vert \sigma_{i}\left(  A^{\left(  rk,sk\right)  }\right)  -\sigma
_{i}\left(  B^{\left(  mk,nk\right)  }\right)  \right\vert \leq\left\vert
\sigma_{i}\left(  A^{\left(  rk,sk\right)  }\right)  -\sigma_{i}\left(
PB^{\left(  mk,nk\right)  }Q\right)  \right\vert \leq\sigma_{1}\left(
A^{\left(  rk,sk\right)  }-PB^{\left(  mk,nk\right)  }Q\right)  .
\]
Now, inequality (\ref{gin1.1}) implies that
\begin{align*}
\left\vert \sigma_{i}\left(  A^{\left(  rk,sk\right)  }\right)  -\sigma
_{i}\left(  B^{\left(  mk,nk\right)  }\right)  \right\vert  &  \leq\sigma
_{1}\left(  A^{\left(  rk,sk\right)  }-PB^{\left(  mk,nk\right)  }Q\right) \\
&  \leq4k\sqrt{2\left\Vert A^{\left(  rk,sk\right)  }-PB^{\left(
mk,nk\right)  }Q\right\Vert _{\square}mnrs}\\
&  \leq6k\sqrt{\widehat{\delta}_{\boxminus}\left(  A^{\left(  rk,sk\right)
},B^{\left(  mk,nk\right)  }\right)  mnrs},
\end{align*}
and so,%
\[
\frac{1}{k\sqrt{mnrs}}\left\vert \sigma_{i}\left(  A^{\left(  rk,sk\right)
}\right)  -\sigma_{i}\left(  B^{\left(  mk,nk\right)  }\right)  \right\vert
\leq6\sqrt{\widehat{\delta}_{\boxminus}\left(  A^{\left(  rk,sk\right)
},B^{\left(  mk,nk\right)  }\right)  }.
\]
Finally, using Proposition \ref{pro2}, we find that
\[
\left\vert \frac{\sigma_{i}\left(  A\right)  }{\sqrt{mn}}-\frac{\sigma
_{i}\left(  B\right)  }{\sqrt{rs}}\right\vert =\frac{1}{k\sqrt{mnrs}%
}\left\vert \sigma_{i}\left(  A^{\left(  rk,sk\right)  }\right)  -\sigma
_{i}\left(  B^{\left(  mk,nk\right)  }\right)  \right\vert \leq6\sqrt
{\widehat{\delta}_{\boxminus}\left(  A^{\left(  rk,sk\right)  },B^{\left(
mk,nk\right)  }\right)  }.
\]
Letting $k$ tend to infinity and passing to limits, the proof is completed.
\end{proof}

\subsection{Proof of Theorem \ref{ssamp}}

\begin{proof}
[\textbf{Proof of Theorem \ref{ssamp}}]Let $X$ be a uniformly random subset of
$\left[  n\right]  $ of size $k.$ Let $B=A\left[  X,X\right]  .$ The result
Borgs et al. implies that
\[
\delta_{\square}\left(  A,B\right)  \leq10\left\vert A\right\vert _{\infty
}\left(  \log_{2}k\right)  ^{-1/2}%
\]
with probability at least $1-\exp\left(  -k^{2}/\left(  2\log_{2}k\right)
\right)  $.

By Cauchy's Interlacing theorem, for every $i=1,\ldots,k,$%
\[
\mu_{i}\left(  A\right)  \geq\mu_{i}\left(  B\right)  \geq\mu_{n-i+1}\left(
A\right)  .
\]
Hence, if $\mu_{i}\left(  B\right)  \geq0,$ then $\mu_{i}\left(  A\right)
\geq0,$ and clause (ii.a) of Theorem \ref{th2} implies that
\[
\left\vert \frac{\mu_{i}\left(  A\right)  }{n}-\frac{\mu_{i}\left(  B\right)
}{k}\right\vert <3\delta_{\square}\left(  A,B\right)  ^{1/2}\leq30\left\vert
A\right\vert _{\infty}\left(  \log_{2}k\right)  ^{-1/4}.
\]
If $\mu_{i}\left(  B\right)  <0,$ then $\mu_{n-i+1}\left(  A\right)  <0,$ and
clause (ii.b) of Theorem \ref{th2} implies that%
\[
\left\vert \frac{\mu_{n-k+i}\left(  G\right)  }{n}-\frac{\mu_{i}\left(
H\right)  }{k}\right\vert \leq3\delta_{\square}\left(  A,B\right)  ^{1/2}%
\leq30\left\vert A\right\vert _{\infty}\left(  \log_{2}k\right)  ^{-1/4},
\]
completing the proof of Theorem \ref{ssamp}.
\end{proof}

\subsection*{Concluding remarks}

\begin{enumerate}
\item Note that the norm $\left\Vert A\right\Vert _{\boxdot}$ seems subtler
than $\left\Vert A\right\Vert _{\square}$. Yet while $\left\Vert A\right\Vert
_{\square}$ was used in a successful version of Szemer\'{e}di's Regularity
Lemma (\cite{FrKa99}), $\left\Vert A\right\Vert _{\boxdot}$ has never been
studied explicitly in this respect. A natural question arises: what type of
Regularity Lemma one can prove using $\left\Vert A\right\Vert _{\boxdot}$.

\item Let $A\in\mathcal{M}_{m,n}$ and $B\in\mathcal{M}_{p,q}.$ We studied the
spectral difference of $A$ and $B$ in the form
\[
\max_{1\leq i\leq\min\left(  p,q,m,n\right)  }\left\vert \frac{\sigma
_{i}\left(  A\right)  }{\sqrt{mn}}-\frac{\sigma_{i}\left(  B\right)  }%
{\sqrt{pq}}\right\vert
\]
On the other hand, the vector of the singular values of a matrix $A$ becomes a
unit vector when divided by $\left\Vert A\right\Vert _{F}.$ Thus, it seems
more appropriate to study
\[
\max_{1\leq i\leq\min\left(  p,q,m,n\right)  }\left\vert \frac{\sigma
_{i}\left(  A\right)  }{\left\Vert A\right\Vert _{F}}-\frac{\sigma_{i}\left(
B\right)  }{\left\Vert B\right\Vert _{F}}\right\vert .
\]
Can Theorem \ref{th3} be modified accordingly? Similar modifications seem
possible for the eigenvalues of Hermitian matrices.

\item Some, but not all, of our results can be extended for complex graphons,
i.e., measurable functions $f:\left[  0,1\right]  ^{2}\rightarrow\mathbb{C.}$
One way of doing this is approximation by step functions with finitely many
steps. We leave these extensions to interested readers.\bigskip
\end{enumerate}

\end{document}